\documentclass{amsart}

\newtheorem{theorem}{Theorem}[section]
\newtheorem{lemma}[theorem]{Lemma}
\newtheorem{corollary}[theorem]{Corollary}
\newtheorem{proposition}[theorem]{Proposition}

\theoremstyle{definition}
\newtheorem{definition}[theorem]{Definition}

\newtheorem{remark}[theorem]{Remark}

\numberwithin{equation}{section}

\begin{document}

\title[The distance function]{The distance function and Lipschitz classes of mappings between metric spaces}

\author{Marijan Markovi\'{c}}

\address{Faculty of Sciences and Mathematics\endgraf
University of Montenegro\endgraf
D\v{z}ord\v{z}a Va\v{s}ingtona bb\endgraf
81000 Podgorica\endgraf
Montenegro\endgraf}
\email{marijanmmarkovic@gmail.com}

\begin{abstract}
We investigate when the local Lipschitz property of the real-valued function $g(z) = d_Y (f(z),A)$ implies the global
Lipschitz property of the mapping $f:X\to Y$ between the metric spaces $(X,d_X)$ and $(Y,d_Y)$. Here, $d_Y(y,A)$ denotes
the distance of $y\in Y$ from the non-empty set $\subseteq Y$. As a consequence, we find that an analytic function on a
uniform domain of a normed space belongs to the Lipschitz class if and only if its modulus satisfies the same condition;
in the case of the unit disk this result is proved by K. Dyakonov. We use the recently established version of a
classical theorem by Hardy and Littlewood for mappings between metric spaces. This paper is a continuation of the recent
article by the author \cite{MARKOVIC.JGA.2}.
\end{abstract}



\thanks{This research was supported by the Ministry of Education, Science, and Innovation of Montenegro
through the grant ''Mathematical Analysis, Optimization, and Machine Learning''.}

\maketitle

\section{Introduction}


This paper is mainly motivated by a result for analytic functions on the unit disk $\mathbb{D}\subseteq\mathbb{C}$
obtained by Dyakonov \cite{DYAKONOV.ACTM}. Among others, he deduced the following interesting result. An analytic
function belongs to the certain Lipschitz class if and only if its modulus belongs to the same Lipschitz class. In
a special case, for $\alpha\in (0,1)$ we have the following two equivalent statements:
\begin{equation*}
f\in \Lambda ^{\alpha} (\mathbb{D}) \Leftrightarrow  |f|\in \Lambda ^{\alpha} (\mathbb{D}).
\end{equation*}
We have denoted by
\begin{equation*}
\Lambda ^\alpha (\mathbb{D}) = \{f:\mathbb {D}\to \mathbb{C}: \exists C_f \text{ such that }
|f(z) - f(w)|\le  C_f|z-w|^\alpha \text{ for } z,w\in \mathbb {D}\}
\end{equation*}
the Lipschitz  (H\"{o}lder) class of functions on the unit disk.


Pavlovi\'{c} \cite{PAVLOVIC.ACTM} gave a new proof of the Dyakonov theorem using the well-known Hardy-Littlewood
theorem which says that an analytic function $f$ on the unit disk satisfies
\begin{equation}\label{EQ.HL.1}
| f'(z) |\le C_f(1- |z|)^{\alpha-1},\quad  z\in \mathbb{D},
\end{equation}
if and only if it is $\alpha$-H\"{o}lder continuous, i.e.
\begin{equation}\label{EQ.HL.2}
|f(z) - f (w)| \le \ C'_f |z-w|^\alpha,\quad z,w\in\mathbb{D}.
\end{equation}
Here, $C_f$ and $C'_f$ are constants that depend only on $f$. Estimates of the type $C'_f \le a C_f$ and
$C_f\le b C'_f$, where $a$ and $b$ are absolute constants, are also known and important.

After the Dyakonov result appeared, many authors considered whether a similar equivalence holds for other mapping
classes, such as quasi-regular mappings or harmonic mappings on domains in $\mathbb{R}^n$. Our aim is to consider this
question for mappings between metric spaces and, in particular, for analytic mappings on domains in a normed space of
not necessarily finite dimension. The main results of this work are collected in the third section after we give some
preliminaries and auxiliary results.

Since a version of the Hardy-Littlewood theorem plays an important role in this paper, we shall present in the sequel
some earlier extensions of it. The results obtained by Gehring and Martio and Lappalainen in a more general setting
are those that served as motivation for the author's generalization of the Hardy-Littlewood theorem for mappings
between metric spaces \cite{MARKOVIC.JGA.2}.

Lappalainen \cite{LAPPALAINEN.AASF} obtained an extension of the one part of the Hardy and Littlewood theorem for
mappings on so called  Lip$_\varphi$-extension domains, where $\varphi:[0,\infty)\to [0,\infty)$ is a majorant
in the meaning: a continuous function, differentiable in $(0,\infty)$, with $\varphi (0) = 0$, $\varphi (t)>0$ for
$t\in (0,\infty)$, an increasing function on $(0,\infty)$, and $\varphi'$ a decreasing function on $(0,\infty)$.
In the definition of Lip$_\varphi$-extension domains, an important role plays local Lipschitz calasses
\cite{LAPPALAINEN.AASF}. Roughly, a domain belongs to this class if every locally Lipschitz function is globally
Lipschitz along with the norm control. A theorem says that a domain $D$ with a non-empty boundary is a
Lip$_\varphi$-extension domain if and only if there exists a constant $M$ such that for every $x, y\in D$ one finds
a rectifiable curve $\gamma_{x,y}\subseteq D$ with ends at $x$ and $y$ such that
\begin{equation}\label{EQ.LAPPALAINEN.COND}
\int _{\gamma_{x,y}} \frac{\varphi (d(z,\partial D))}{d(z,\partial D)} |dz|  \le M  \varphi ( |x-y|),
\end{equation}
where $d(z,\partial D)$  is the distance of $z$ from $\partial  D$. In the case of standard majorants
$\varphi(t) = \varphi_\alpha (t)=t^\alpha$, $\alpha\in (0,1)$, it is said that a domain is a Lip$_\alpha$-extension
domain, following the terminology of Gehring and Martio.

A variant of the one part of the Hardy and Littlewood theorem obtained by Lappalainen
\cite[Theorem 7.3]{LAPPALAINEN.AASF} states that if an analytic function on a Lip$_\varphi$-extension domain
$D\subseteq \mathbb{C}$ satisfies
\begin{equation}\label{EQ.LAPPALAINEN.1}
| f'(z) |\le  C_f\frac{\varphi (d (z,\partial D))}{d(z,\partial D)},\quad  z\in D,
\end{equation}
then $f$ satisfies the Lipschitz type condition
\begin{equation}\label{EQ.LAPPALAINEN.2}
|f(z) - f (w)| \le C'_f \varphi ( |z-w|),\quad z,w\in D,
\end{equation}
where $ C_f$ and $ C'_f$  are constants. Gehring and Martio obtained a generalization of the same part of the Hardy
and Littlewood theorem  for mappings on uniform domains \cite{GEHRING.CV}, and somewhat later on Lip$_\alpha$-extension
domains \cite{GEHRING.AASF}.

In this paper a variant of the Gehring and Martio result is proved for analytic mappings on uniform domains in a normed
space.

\section{Preliminaries}


A continuous mapping $\gamma:[a,b]\to X$, is called a curve in a metric space $(X,d_X)$ that connects $\gamma (a)$ and
$\gamma (b)$. For simplicity, it is convenient to identify the curve $\gamma$ with its image $\gamma([a,b])$. In case
$\gamma$ being a 1-1 mapping, the curve is called an arc. The length of the curve $\gamma$ is given by
\begin{equation*}
\ell  (\gamma)   = \sup   \sum_{i=1}^n d_X  (\gamma(t_{i-1}),\gamma (t_i)),
\end{equation*}
where the supremum is taken over all finite partitions $a = t_0 < t_1 < \dots < t_n = b$ of $[a,b]$. A curve is
rectifiable if its length is finite. The family of rectifiable curves that connect $x\in X$ and $y\in X$ is denoted by
$\Gamma({x,y})$. A metric space $X$ is said to be rectifiable connected if $\Gamma(x,y)$ is not empty for every
$x, y\in X$. If $\gamma:[a,b]\to X$ is a rectifiable curve in a metric space $X$, and $[c,d]\subseteq [a,b]$, then the
(rectifiable) curve $\gamma:[c,d]\to X$ is denoted by $\gamma_{[c,d]}$. In case $\gamma$ is an arc, $\gamma(c) = z$,
and $\gamma (d) = w$, we shall also use the notation $\gamma [z,w]$ for this part of the curve.

The integral of a continuous function $f$ on $X$ over a rectifiable curve $\gamma\subseteq X$ is denoted by
$\int_\gamma f$. It may be defined via generalized Riemann sums, as in \cite{MARKOVIC.JGA.1}:
For every $\varepsilon>0$ there exists $\delta>0$ such that
\begin{equation*}
\left|\int_\gamma  f - \sum _{i=1}^n f (\gamma (s_i)) \ell(\gamma|_{[t_{i-1},t_i]} ) \right|<\varepsilon
\end{equation*}
for every subdivision $a=t_0<t_1<\cdots<t_n= b$ of $[a,b]$ with $\max_{1\le i\le n}|t_i-t_{i-1}|<\delta$, and every
sequence $s_1,s_2,\dots,s_n$ with $t_{i-1}\le s_i\le t_i$, $1\le i \le n$.

This definition is convenient for our consideration in the papers \cite{MARKOVIC.JGA.1} and \cite{MARKOVIC.JGA.2}.
On the other hand, the definition may be given as in \cite[Chapter V]{HEINONEN.BOOK} not only for continuous functions,
using the fact that every rectifiable curve can be parameterized by the arc length. Let $\gamma$ be a rectifiable curve
and denote its length by $\ell$. For a Borel measurable function $f$ in a metric space $X$ one can define the integral
of $f$ over this curve by
\begin{equation*}
\int _\gamma f  = \int_0^\ell f(\gamma(s))ds,
\end{equation*}
where on the right side we have the standard Lebesgue integral.
 
An everywhere positive and continuous function on a metric space is called a weight function. Let $w$ be a weight
on a rectifiable connected metric space $X$. The weighted distance on $X$ generated by this weight is given by
\begin{equation*}
d_w   (x, y)  =  \inf_{\gamma \in \Gamma(x,y)} \int_\gamma w, \quad x,y\in X.
\end{equation*}

An important example occurs considering a domain $D$ in a metric space $X$ with non-empty boundary. In this case, the
weight function on $D$ can be defined by setting $w (x) = d_X(x, \partial D)^{-1}$, $x\in D$. Since the set $D$ is
open, we clearly have the positivity of $w$. The distance $d_w$ is called the quasi-hyperbolic distance on $D$ and
plays a very important role in some considerations. If $w$ is equal to $1$ everywhere, we obtain the so-called inner
distance on $D$.

\medskip


We now introduce the mapping classes that will appear in this work.

In the sequel, let $(X,d_X)$ and $(Y,d_Y)$ be two metric spaces. For $\varphi:[0,\infty)\to [0,\infty)$, which we shall
call the majorant, the (global) Lipschitz mapping class $\Lambda^\varphi(X,Y)$ is defined in the following way. A
mapping $f:X\to Y$ belongs to this class if there exists a constant $C_f$ such that
\begin{equation*}
d_Y (f(x),f(y)) \le  C_f  \varphi (d_X (x,y)),\quad  x,y\in X.
\end{equation*}
The infimum of $C_f$ in all $x,y\in X$ is called the norm of $f$ in the class $\Lambda^\varphi (X,Y)$, and it is denoted
by $\|f\|_{\Lambda^\varphi (X,Y)}$. For $\varphi = \varphi_\alpha$, where $\varphi_\alpha (t) = t^\alpha$,
$t\in [0,\infty)$, $\alpha\in (0, \infty)$, we write $\Lambda ^\alpha (X,Y)$ instead of $\Lambda^{\varphi_\alpha}(X,Y)$.
In the Introduction, we have encountered the class $\Lambda ^\alpha (\mathbb{D})$ which is
$\Lambda ^\alpha (\mathbb{D},\mathbb{C})$ in our present notation.

The local Lipschitz class with respect to a weight $w$ on the metric space $X$ is indicated by
$\Lambda^\varphi_{w} (X,Y)$. It contains mappings $f:X\to Y$ for which there exists a constant $C_{f}^w$ such that
\begin{equation*}
d_ Y (f(x) ,  f(y)) \le C_{f}^w \varphi (  d_X (x,y)), \quad y\in B(x, w(x)).
\end{equation*}
The infimum of all $C_{f}^w$ in all $x\in X$ and $y\in B(x, w(x))$ is defined as the norm in this class of mappings. As
before, if $\varphi = \varphi_\alpha$ we write $\Lambda ^{\alpha}_w(X,Y)$ instead of $\Lambda ^{\varphi_\alpha}_w(X,Y)$.
When considering the local Lipschitz class on a domain $D\subseteq X$ with nonempty boundary, we write
$\Lambda ^\varphi_{\mathrm {loc}} (D, Y)$ in the case where the weight function is given via the distance function
$w(x) =\frac 12 d(x,\partial D)$, $x\in D$.

For a mapping $f:X\to Y$, define
\begin{equation*}
{d}^\ast f (x) = \limsup _{y\to x} \frac {d_Y ( f(x), f(y))}{d_X(x,y)},
\end{equation*}
if $x\in X$ is not isolated. We set $d^\ast f (x) = 0$, if $x$ is isolated. The class $\mathrm {D}^\ast(X,Y)$
contains mappings $f:X\to Y$ such that $d^\ast f (x)$ is finite for every $x\in X$.

The Bloch class $\mathrm{B}_{w}(X,Y)$ consists of mappings $f\in \mathrm {D}^\ast(X,Y)$ for which there exists a
constant $C_f'$ such that
\begin{equation*}
d^\ast f(x) \le C_f'  w(x),\quad  x\in X.
\end{equation*}
The least possible constant $C_f'$ in the above inequality is defined as the norm of $f$ in the class, and it is
denoted by $\|f\|_{\mathrm {B}_{w} (X,Y)}$.

\medskip


We recall some facts regarding analytic mappings on domains in a normed space and list some useful lemmas with
references. For a version of the classical Schwarz-Pick inequality, we provide a sketch of the proof as follows
\cite{BLASCO.JFA} and \cite{KALAJ.GJM}.

A mapping $f:D\to Y$, where $D \subseteq X$ is a domain in a normed space $X$ (over $\mathbb{R}$ or $\mathbb{C}$), and
$Y$ is another normed space (over the same field), is said to be Fr\'{e}chet differentiable at $x\in D$, if there exists
a continuous linear mapping $d f(x):X\to Y$, called Fr\'{e}chet differential, such that in an open ball
$B_X(x,r) = \{y\in X: \|x-y\|_X < r\}\subseteq D$ there holds
\begin{equation*}
f(y) -  f(x) = d f (x)  (y-x) +  \alpha (y-x),\quad y\in B_X(x,r),
\end{equation*}
where $\alpha:B_X (0,r)\to  Y$ satisfies $\lim _{h\to  0} \frac {\|\alpha (h)\|_Y}{\|h\|_X} = 0$. We denote by
\begin{equation*}
\|d f (x)\|_{X\to Y} = \sup _{\|\zeta\|_X = 1} \|d f(x)\zeta\|_Y
\end{equation*}
the norm of the linear mapping $d f(x): X\to Y$. That norm coincides with $d^\ast f(x)$. The proof of this fact can be
found in \cite{MARKOVIC.JGA.1}. More precisely, we have

\begin{lemma}\label{LE.FRECHET}
Let $X$ and $Y$ be normed spaces over the same field ($\mathbb{R}$ or $\mathbb{C}$) and let $D\subseteq X$ be a domain.
If a mapping $f:D\to Y$ is Frech\'{e}t differentiable at $x\in D$, then we have
\begin{equation*}
\| d f (x) \|_{X\to Y} =  \limsup_{y\to x} \frac {\|f(x) - f(y)\|_Y}{\|x - y\|_X}.
\end{equation*}
\end{lemma}

If $X$ and $Y$ are complex normed spaces, then a mapping $f:D\to Y$ is said to be analytic if it is Fr\'{e}chet
differentiable everywhere on the domain $D$. The well-known result for bounded analytic functions on the unit disk
$\mathbb{D}$ is the Schwarz-Pick lemma. We shall need an extension of this result for analytic mappings on the unit ball
in a normed space that is not necessarily finite-dimensional. This is the content of the following lemma.

\begin{lemma}[The Schwarz-Pick lemma]\label{LE.SCHWARZ-PICK}
Let $X$ and $Y$ be complex normed spaces and assume that the norm on $Y$ is generated by an inner product. Let
$f:B_X(0,1)\to Y$ be an analytic mapping bounded by $1$. Then we have
\begin{equation*}
\|d f(0)\|_{X\to Y}\le {1-\|f(0)\|_Y^2},\quad  \text{if }\dim Y =1,
\end{equation*}
and
\begin{equation*}
\|d f(0)\|_{X\to Y}\le \sqrt{1-\|f(0)\|_Y^2},\quad \text{if }\dim Y\not = 1.
\end{equation*}
\end{lemma}

In his work, Kalaj \cite{KALAJ.GJM} gave an interesting proof of this result in the finite-dimensional case, i.e., for
$X = \mathbb{C}^n$ and $Y = \mathbb{C}^m$, using the M\"{o}bius transforms in $\mathbb{C}^n$. However, this proof can
be extended in the setting of the above lemma.

As in the lemma, let $Y$ be an inner product space over the field of complex numbers. The  M\"{o}bius transforms of the
unit ball $B_Y(0,1)$ onto itself are analytic mappings given by
\begin{equation*}
\varphi_a  (z) = \frac{a - P_a z - s_a Q_a z}{1- \left<z,a\right>},\quad z\in B_Y(0,1),
\end{equation*}
where for $a\in B_Y(0,1)$ we have denoted $s_a = \sqrt{1-\|a\|_Y^2}$, $P_a:Y\to Y$ is the orthogonal projection along
the one-dimensional subspaces, that is,
\begin{equation*}
P_a(z) = \frac{\left<z,a\right>}{\left<a,a\right>}a, \quad z\in Y,
\end{equation*}
in case $a\ne  0$ ($P_ 0 (z) = 0$), and $Q_a:Y\to Y$, is the orthogonal complement of $P_a$, i.e. $Q_a = \mathrm {Id} - P_a$.

It is easy to verify that $\varphi_a (0) = a$. Moreover, $\varphi_a$ is an involutive mapping, which means that
$\varphi_a \circ\varphi_a = \mathrm{Id}$ (identity) and therefore $\varphi^{-1} = \varphi$.

In \cite[Lemma 3.2]{BLASCO.JFA}, the authors show that $d\varphi_a (0)$ is a linear combination of $P_a$ and $Q_a$,
that is,
\begin{equation*}
d \varphi_a (0) = - s_a^2P_a - s_a Q_a.
\end{equation*}
In \cite[Lemma 2.2]{KALAJ.GJM} the author finds the norm of the continuous linear mapping $d\varphi_a (0)$ in the
finite-dimensional setting. It depends on the dimension of the space $Y$ in the following way:
$\|d \varphi_a (0)\|_{Y\to Y} = 1 - \|a\|_Y^2$, if $\dim Y = 1$, and $\|d \varphi_a (0)\|_{Y\to Y} = \sqrt{1-\|a\|^2_Y}$,
when the dimension of $Y$ is not equal to $1$. Although the last result is derived for finite-dimensional $Y$, the
proof is valid without this assumption.

\begin{proof}[Proof of Lemma \ref{LE.SCHWARZ-PICK}]
A special case of the lemma is when we have a mapping $g$ that maps $0_X$ to $0_Y$ in place of $f$. In this case, the
inequality $\|dg(0)\|_{X\to Y} \le  1$ follows from the ordinary Schwarz lemma. In fact, we can consider the analytic
mapping $g_\zeta (z) = (y^\ast\circ  g\circ u_\zeta )(z)$ on the unit disk $\mathbb{D}$, where $y^\ast$ is a linear
functional on $Y$ of the norm $1$, $\zeta\in X$ is of the norm $1$, and $u_\zeta (z) = z \zeta$, $z\in \mathbb{D}$.
Since $|g_\zeta (z)|\le 1$, $z\in \mathbb{D}$, and $g_\zeta (0) = 0$, we can apply the ordinary Schwarz lemma. By this
classical result, we have $|g'_\zeta (0)|\le 1$ which, bearing in mind that $g_\zeta$ is the composition of analytic
mappings, implies $|y^\ast ( dg (0)\zeta ) |\le 1$. Since this holds for every $y^\ast\in Y^\ast$ and $\zeta$ of the
unit norm, we first conclude that $\|dg (0)\zeta\|_Y\le 1$, and finally $\|dg(0)\|_{X\to Y}\le 1$.

Let us now consider the general case. Let $f:B_X(0,1)\to Y$ be as in the formulation of this lemma and denote
$a = f(0_X)$. The general case follows from the special case by applying a certain M\"{o}buis transform on the
mapping $f$. In fact, let $g   = \varphi _a \circ f$. Then we have $g(0_X) = \varphi_a(f(0_X)) = \varphi_a (a) = 0_Y$.
Applying the Schwarz lemma to the mapping $g$, we find $\|dg (0)\|_{X\to Y} \le 1$. Since the mapping $\varphi_a$
is involutive, we obtain $f = \varphi _a \circ g$ and $df(0_X) = d\varphi _a(0_Y) \circ dg(0_X)$. It follows
\begin{equation*}
\|df(0)\|_{X\to Y} \le \|d\varphi _a(0_Y) \|_{Y\to Y}  \|dg(0_X)\| _{X\to Y} \le \|d\varphi _a(0) \|_{Y\to Y},
\end{equation*}
which we aimed to prove.
\end{proof}

\section{The main results}


The following two results for mappings between metric spaces, proved by the author \cite{MARKOVIC.JGA.2}, generalize
the Hardy-Littlewood theorem stated in the Introduction. For the first result, we need not make any assumptions about
the behavior of the $d^\ast f$. For the second result, we need the well local behavior of $d^\ast f$, which we clarify
in the following definition. For this converse result, we also need some stronger conditions for the majorant.

\begin{definition}
Assume that $w$ is a weight on a metric space $X$, and let $Y$ be another metric space. A mapping
$f\in\mathrm {D}^\ast(X,Y)$, bounded on every open ball $B_X(x,r)$, $x\in X$, $0<r<w(x)$, is said to be regularly
oscillating if there exists a constant $K$ such that
\begin{equation}\label{EQ.RO}
d^\ast f  (x)  \le  \frac  {K} r {\sup_{y\in B_X(x,r)} d_Y ( f(x),  f(y) )},\quad x\in X,\,  0<r<w(x).
\end{equation}
\end{definition}

Note that the right-hand side of \eqref{EQ.RO} resembles the maximal function.

For example, a bounded analytic mapping on a domain $D$ with a nonempty boundary in a normed space belongs to the class
of regularly oscillating mappings with respect to the weight $w(z) = d(z,\partial D)$, $z\in D$, as proved in
\cite{MARKOVIC.JGA.2} using the Schwarz lemma. In this case, the relation \eqref{EQ.RO} is valid with the constant $K=1$.

\begin{proposition}\label{PR.1}
Let $(X,d_X)$ and $(Y,d_Y)$ be two metric spaces, $w$ a weight on the first one, and $\varphi:[0,\infty)\to [0,\infty)$
a majorant, such that the following condition holds
\begin{equation}\label{EQ.COND.LAPP.VARPHI}
(\exists M>0)(\forall x,y\in X)(\exists\gamma_{x,y}\in \Gamma(x,y)) \int _{\gamma_{x,y}} w \le M \varphi (d_X (x,y)).
\end{equation}
Then we have
\begin{equation*}
\|f\| _{ \Lambda^\varphi  (X,Y)}\le   M\|f\| _{ \mathrm {B}_w (X,Y)}, \quad f\in \mathrm {B}_w (X,Y).
\end{equation*}
\end{proposition}

In other words, the proposition says that if $f\in \mathrm {B}_w (X,Y)$ satisfies $d ^\ast  f(x) \le C_f w(x)$, $x\in X$,
where $C_f$ is a constant, then $f$ belongs to $\Lambda^\varphi (X,Y)$, and we have
$d_Y(f(x), f(y))\le  {C}' _f \varphi (d_X(x,y))$,  $x, y\in X$ with ${C}'_f = M C_f $.

\begin{remark}
The integral condition \eqref{EQ.COND.LAPP.VARPHI} connects the weight on a metric space with its ordinary metric
through the majorant. The condition implies that the weighted distance may be controlled by the ordinary one of the
metric space. Instead of that condition, we could assume the condition which more explicitly concerns the weighted
distance 
\begin{equation*}
d_w (x,y) \le M \varphi ( d_X(x,y)),\quad x,y\in X.
\end{equation*}
For example, the condition \eqref{EQ.COND.LAPP.VARPHI} is satisfied by uniform domains and the distance function as
a weight and the standard majorants $\varphi_\alpha$, $\alpha \in (0,1)$. This is the content of the simple lemma
that follows.
\end{remark}

\begin{proposition}\label{PR.2}
Let the majorant $\varphi:[0,\infty)\to [0,\infty)$ be in the class $C^1 (0,\infty)$ with the following property:
There exists a constant $A>0$ such that
\begin{equation}\label{EQ.COND.A}
\frac  {\varphi (t)} {t} < A \varphi'(t), \quad t\in (0,\infty).
\end{equation}
Then for a  mapping $f$ which satisfies \eqref{EQ.RO} we have
\begin{equation*}
\|f\| _{\mathrm {B} _{\varphi'\circ w} (X,Y)}\le A K \|f\|_{\Lambda^\varphi(X,Y)},\quad f\in {\Lambda^\varphi(X,Y)},
\end{equation*} 
for every weight  $w$  on the metric space $X$.
\end{proposition}


As a consequence of the above two propositions, we can derive a version of the Hardy-Littlewood theorem for analytic
mappings on a uniform domain in a normed space. Recall that a domain $D$  with a nonempty boundary in a normed space
over $\mathbb{R}$ is called $c$-uniform if for every $x,y\in D$ there exists a rectifiable arc $\gamma\subseteq D$ with
endpoints at $x$ and $y$ such that:

(i) $\ell(\gamma) \le c \|x-y\|_X$;

(ii) $\min\{\ell(\gamma[x,z]), \gamma[z,y]\} \le c d(z,\partial D)$, $z\in \gamma$.

For example, the unit ball in a normed space is a $2$-uniform domain \cite{VAISALA.CGD}.


\begin{lemma}\label{LE.UNIFORM.ALPHA}
A $c$-uniform domain in a normed space satisfies \eqref{EQ.COND.LAPP.VARPHI} in the case of the standard majorant
$\varphi = \varphi_\alpha$, $\alpha\in (0,1)$, with respect to the weight function $w(x) = d(x,\partial D)^{\alpha-1}$
on $D$ with the constant $M = \frac {2c}\alpha$. In particular, for the unit ball we have $M = \frac {4}\alpha$.
\end{lemma}

\begin{proof}
For $x$, $y\in D$ let $\gamma$ be a rectifiable arc that satisfies the conditions for uniform domains. Let it be
parametrized by the arc length, and let its length be denoted by $\ell$.

Having in mind that $\alpha -1<0$, and  $\ell \le c \|x-y\|^\alpha_X$, we obtain
\begin{equation*}\begin{split}
\int_\gamma d(z,\partial D)^{\alpha-1} &
\le c^{1-\alpha} \int_\gamma\min\{\ell(\gamma[x,z]),\ell( \gamma[z,y])\}^{\alpha-1}
\\&=2c^{1-\alpha}\int_0^{\frac \ell2}s^{\alpha -1}ds = \frac {2c^{1-\alpha}}\alpha\left(\frac l2\right)^\alpha
\\&\le \frac {2 c^{1-\alpha}}\alpha \frac {c^\alpha\|x-y\|^\alpha_X} {2^\alpha} \le  \frac {2c}{\alpha} \|x-y\|^\alpha_X,
\end{split}\end{equation*}
which proves that the condition  \eqref{EQ.COND.LAPP.VARPHI}  is satisfied for $M = \frac {2c}\alpha$.
\end{proof}

\begin{theorem}
For $\alpha\in (0,1)$ an analytic mapping $f$ on a $c$-uniform domain $D$ in a complex normed space $X$ over, with
the range in another complex normed space $Y$, satisfies
\begin{equation*}
\|df (z)\|_{X\to Y} \le C_f d(z,\partial D)^{\alpha-1}, \quad z\in D,
\end{equation*}
if and only if $f$ is $\alpha$-H\"{o}lder continuous, i.e.
\begin{equation*}
\|f(z) - f(w)\|_Y\le C'_f\|z-w\|_X^{\alpha},\quad  z,w\in D,
\end{equation*}
where $C_f$ and $C'_f$ are constants. Moreover, we have estimates $C'_f\le \frac{2c}\alpha C_f$ and $C_f\le C'_f$.
\end{theorem}

\begin{proof}
One part of this theorem follows from Proposition \ref{PR.1} having in mind Lemma \ref{LE.FRECHET} and the preceding
lemma. The converse part follows from Proposition \ref{PR.2} since the majorant $\varphi_\alpha$ satisfies the condition
there for the constant $A = \frac{\beta}{\alpha}$, where $\beta>1$ is a real number. Letting $\beta\to 1$ we obtain the
converse estimate concerning the constants.
\end{proof}

Therefore, the classic theorem of Hardy and Littlewood remains for analytic mappings on uniform domains in a normed
space of not necessary finite dimension. Moreover, mutual constant estimates are universal and depend only on $c$ and
$\alpha$. In the case of analytic mappings on the unit ball in a normed space, we can set $a = \frac {4}{\alpha}$ and
$b = 1$, where $a$ and $b$ are the constants from the formulation of this classical theorem in the Introduction. The
preceding theorem also generalizes the result of Gehring and Martio \cite{GEHRING.CV}.


\medskip

Our next result concerns a new subclass of $\mathrm {D}^\ast (X,Y)$ which we call the class of $p$-regular mappings.
We give a precise definition of this notion in the following.

\begin{definition}\label{DEF.P-REGULAR}
Let $X$ and $Y$ be metric spaces and let $w$ be a weight on $X$. A mapping $f\in\mathrm {D}^\ast(X,Y)$, bounded on
every open ball $B_X(x,r)$, $x\in X$, $0<r<w(x)$, is said to be $p$-regular, $p\in [1,\infty)$, with respect to the
weight $w$, if there exist a constant $K$ and a non-empty set $A\subseteq Y$ such that the following local condition,
which resembles the classical Schwarz-Pick inequality, is satisfied
\begin{equation*}
d^\ast f (x) \le \frac {K}r 
\sup_{y\in B(x,r)} |d_Y(f(x), A) ^p  -  d _Y(f(y), A)^p|^{\frac 1p},
\quad x\in X,\, 0<r<w(x),
\end{equation*}
where   $d_Y(z, A) = \inf_{w\in A} d_Y(z,w)$ is the distance  of $z$ from the set $A$.
\end{definition}

Since in a general metric space we do not have the norm, we have used the distance function from a set $A$.

One of the main results of this paper is given in the following theorem. It also contains a version of the condition
\eqref{EQ.COND.LAPP.VARPHI}. The reason for including this condition is the ability to apply the Hardy-Littlewood
theorem given in Proposition \ref{PR.1}.

\begin{theorem}\label{TH.MAIN}
Let $X$ and $Y$ be two metric spaces and $w$ a weight on $X$. Assume that $f:X\to Y$ is $p$-regular as in
Definition \ref{DEF.P-REGULAR}. Then, if $g(z) = d ( f(z), A )^p$ belongs to the local Lipschitz class
$\Lambda^\alpha_{w} (X,\mathbb {R})$, $\alpha \in (0,\infty)$, then $f$ belongs to the class
$\Lambda^{\frac\alpha p} (X,Y)$, provided that the following condition is satisfied: For $\beta = \frac {\alpha}p$ there
exists a constant  $M_\beta$, such  that
\begin{equation*}
d_{w^{\beta-1}} (x,y)\le M_\beta d_X (x,y)^\beta,\quad x,y\in X.
\end{equation*}
Moreover, we have
\begin{equation*}
\|f\| _{\Lambda^{\frac \alpha p} (X,Y)}\le M_{\frac \alpha p}K \|g\| _{\Lambda^\alpha_{w} (X,\mathbb {R} )} ^\frac 1p.
\end{equation*}
\end{theorem}

\begin{remark}
Note that for $f\in \Lambda^\alpha (X,Y)$ we have $g\in \Lambda^\alpha (X,\mathbb {R})$, which follows immediately from
the triangle inequality, i.e.
\begin{equation*}
|d_Y(f(x),A) - d_Y(f(y),A) |\le d_Y(f(x), f(y)), \quad x,y\in X.
\end{equation*}
This holds without any assumption on the metric space $X$ (such as the existence of a weight on it). Our result in
Theorem \ref{TH.MAIN} gives some sufficient conditions for the converse implication with possibly different exponents of
the Lipschitz classes.
\end{remark}

\begin{proof}[Proof of Theorem \ref{TH.MAIN}]
Since $g(x) = d_Y(f(x), A)^p$ is in the local class $\Lambda^\alpha_{w} (X,\mathbb {R})$, there exists a constant
$C = C_{f}^w $ such  that for every $x\in X$ we have
\begin{equation*}
| d_Y  (f(x), A)^p  - d_Y (f(y), A)^p| \le C d_X(x,y)^\alpha,\quad y\in B(x,w(x)).
\end{equation*}
Since $f$ is $p$-regular mapping with respect to the weight $w$ and the set $A$, applying the preceding inequality, we
obtain that for $r<w(x)$ there holds
\begin{equation*}\begin{split}
d^\ast f (x) & \le \frac {K}r \sup_{y\in B(x,r)} | d_Y (f(x), A) ^p  -  d_Y(f(y), A)^p|^{\frac 1p}
\\& \le \frac {K}r \sup_{y\in B(x,r)} C^{\frac 1p} d_X(x,y)^{\frac \alpha p}
 \le \frac {K}r  C^{\frac 1p} r^{\frac \alpha p}
 =   K  C^{\frac 1p} r^{\frac \alpha p-1}.
\end{split}\end{equation*}
Since the last  estimate holds for every positive  $r<w(x)$, it follows
\begin{equation*}
d^\ast f (x) \le   K  C^{\frac 1p} w(x)^{\frac \alpha p-1}, \quad  x\in X.
\end{equation*}
Now, we apply Proposition \ref{PR.1} with $w^{\frac \alpha p -1}$ as a weight on the metric space $X$ and the standard
majorant $\varphi = \varphi_{\frac \alpha p }$. We conclude that $f$ belongs to the class
$\Lambda^{\frac \alpha p} (X,Y)$. Moreover, by the same theorem, we have
\begin{equation*}
d _ Y( f(x), f(y))\le M_{\frac \alpha p}  K C^{\frac 1p} d_X(x,y)^{\frac \alpha p}, \quad x,y\in X.
\end{equation*}
This inequality implies the norm estimate
\begin{equation*}
\|f\|_{ \Lambda^{\frac \alpha p}  (X,Y)}\le M_{\frac \alpha p} K  \|g\|_{\Lambda_{w}^\alpha (X,\mathbb{R})}^{\frac 1p},
\end{equation*}
which we aimed to prove.
\end{proof}


We shall now apply the result given in Theorem \ref{TH.MAIN} to bounded analytic mappings on uniform domains in a normed
space. In particular, this theorem has a consequence for bounded analytic mappings on the unit ball in a normed space.
Its application will produce two corollaries depending on the range space. Their proofs are also based on the following
lemma, which says that bounded analytic mappings are 1-regular or 2-regular, depending on the dimension of the range
space.

\begin{lemma}
A bounded analytic mapping $f:D\to Y$, where $D$ is a domain with a non-empty boundary in a complex normed space $X$,
and $Y$ is a complex inner product space, is 1-regular if $\dim(Y)=1$ and 2-regular if $\dim Y$ is not $1$. We assume
regularity in the sense of Definition \ref{DEF.P-REGULAR} with the distance function from the boundary of the domain
$D$ as a weight and the set $A = \{0_Y\}$. In both cases, we have $K=1$.
\end{lemma}

\begin{proof}
Let $x\in D$. For $r<d(x, \partial D)$ the mapping
\begin{equation*}
g( z ) = M ^{-1}f(x + rz),\quad M = \sup_{w\in B(x, r)}\|f(w)\|_Y,
\end{equation*}
is analytic on the unit ball $B_X(0,1)$ and bounded by $1$. By the Schwarz-Pick inequality in Lemma \ref{LE.SCHWARZ-PICK}
we obtain
\begin{equation*}
\|dg (0)\|_{X\to Y}\le \sqrt{1 -\|g(0)\|^2_Y }
\end{equation*}
in the case $\dim Y\not  = 1$. Since  $d g(0) =  r M ^{-1}df(x)$, by the last inequality we have
\begin{equation*}
\|d f (x) \|_{X\to Y} \le \frac {M}r \sqrt{1-\|g(0)\|^2_Y}
=  \frac {M}r \sqrt{1- \frac {\|f(x)\|^2_Y}{M^2}}
= \frac 1r \sqrt{ M^2 -  \|f(x)\|^2_Y} .
\end{equation*}
It follows
\begin{equation*}
\|d f (x) \|_{X\to Y} \le \frac 1 r \sup_{w\in B_X(x,r)} | \|f(w) \|^2_Y  - \| f(x)\|_Y^2|^{\frac 12},
\quad x\in D,\, 0<r<d(x,\partial D),
\end{equation*}
which proves  that  $f$  is $2$-regular.

In a similar fashion, one shows that a bounded analytic function is $f:D \to \mathbb {C}$ is $1$-regular. In this case,
the Schwarz-Pick inequality says that
\begin{equation*}
\|d g (0)\|_{X\to Y}\le  1 -\|g(0)\|^2_Y,
\end{equation*}
and the rest of the proof is similar.
\end{proof}

\begin{corollary}\label{CORO.NORMED.DIM.1}
Let $f:D\to \mathbb{C}$ be a bounded analytic function on the $c$-uniform domain $D$ in a complex normed space $X$. If
$g = |f|$ belongs to the class $\Lambda^{\alpha}_{\mathrm {loc}}(D,\mathbb{R})$, then $f$ belongs to
$\Lambda^{\alpha}(D,\mathbb {C})$, and we have
\begin{equation*}
\|f\|_{\Lambda^{\alpha}(D,\mathbb{C})} \le \frac{4c}\alpha\|g\|_{\Lambda^{{\alpha}}_{\mathrm {loc}}(D,\mathbb{R})}.
\end{equation*}
\end{corollary}

\begin{corollary}\label{CORO.NORMED.DIM,N>1}
Let $X$ be a normed space, $Y$ an inner product space, both over the complex field, and $f:D\to Y$ be a bounded analytic
mapping on the $c$-uniform domain $D\subseteq X$. If the function $g = \|f\|_Y^2$ belongs to
$\Lambda^{\alpha}_{\mathrm{loc}}(D,\mathbb{R})$, then $f$ belongs  $\Lambda^{\frac{\alpha}2}(D,Y)$, and we have
\begin{equation*}
\|f\|_{\Lambda^{\frac{\alpha}2}(D,Y)} 
\le\frac{8c}\alpha\|g\|_{\Lambda^{{\alpha}}_{\mathrm {loc}}(D,\mathbb{R})}^{\frac 12}.
\end{equation*}
\end{corollary}

\begin{proof}[Proof of Corollary \ref{CORO.NORMED.DIM.1} and Corollary \ref{CORO.NORMED.DIM,N>1}]
Both corollaries follow from Theorem \ref{TH.MAIN}. First of all, by the above lemma we find that a bounded analytic
function on a domain in a normed space is 1-regular and that an analytic mapping is 2-regular provided that the
dimension of the range space is grater $1$. In both cases, we have $K = 1$.

Recall that for the local Lipschitz class in the domain $D$, the weight function is given by
$w(x) = \frac 12 d(x, \partial D)$, $x\in D$. Taking into account Lemma \ref{LE.UNIFORM.ALPHA}, it is easy to check
that this weight function satisfies the condition of Theorem \ref{TH.MAIN} with constant
$M_\beta = \frac {4c}{\beta}$ for every $\beta\in (0,1)$ (just dividing the inequality by $2^{\beta-1}$ and using
some simple estimates). Therefore, as a direct consequence of Theorem \ref{TH.MAIN}, in case $\dim Y \not =  1$, we
have the first one corollary, and in the case $\dim Y \not = 1$ we have the second one.
\end{proof}

\end{document}